\documentclass[reqno]{amsart}

\usepackage{graphicx}
\usepackage{amssymb}
\usepackage{cases}
\usepackage[sort&compress]{natbib}\setcitestyle{numbers,square,comma}
\usepackage[unicode,pagebackref,colorlinks,urlcolor=blue,citecolor=blue]{hyperref}
\numberwithin{equation}{section}
\numberwithin{figure}{section}
\newtheorem{theorem}{Theorem}[section]
\newtheorem{conjecture}[theorem]{Conjecture}
\newtheorem{definition}[theorem]{Definition}
\newtheorem{lemma}[theorem]{Lemma}
\newtheorem{corollary}[theorem]{Corollary}
\DeclareMathOperator{\IM}{Im}
\DeclareMathOperator{\LT}{\textup{\textsc{lt}}}
\DeclareMathOperator{\spn}{span}

\DeclareMathOperator{\cf}{cf}

\newcommand{\cfx}[2]{\cf_{K[y]}(#1,x^{#2})}
\newcommand{\id}{{\rm id}}

\begin{document}
\title[Images of locally finite $\mathcal{E}$-derivations]{Images of locally finite $\mathcal{E}$-derivations of bivariate polynomial algebras}
\author{Hongyu Jia, Xiankun Du and Haifeng Tian}
\address{X. Du: School of Mathematics,  Jilin university, Changchun 130012, P. R. China}
\email{duxk@jlu.edu.cn}
\address{H. Jia: School of Mathematics,  Jilin university, Changchun 130012, P. R. China}
\email{jiahy1995@163.com}
\address{H. Tian: School of Mathematics,  Jilin university, Changchun 130012, P. R. China}
\email{tianhf1010@foxmail.com}
\begin{abstract}
This paper presents an $\mathcal{E}$-derivation analogue of a result on
derivations due to van den Essen, Wright and Zhao. We prove that the image
of a locally finite $K$-$\mathcal{E}$-derivation of polynomial algebras in two
variables over a field $K$ of characteristic zero is a Mathieu subspace.
This result together with that of van den Essen, Wright and Zhao confirms the
LFED conjecture in the case of polynomial algebras in two variables.
\end{abstract}
\thanks{Corresponding author: Haifeng Tian}
\maketitle

\section{Introduction}

Images of derivations have been studied recently by several authors because of close relationship with the Jacobian conjecture. Let $K$ denote a field
of characteristic $0$. It is proved in~\cite{vanden2010} that the Jacobian
conjecture for $K[x,y]$ is equivalent to the statement that the image $\IM D$
is a Mathieu subspace of $K[x,y]$ for any $K$-derivation $D$ of $K[x,y]$ such
that $1\in\IM D$ and $\operatorname{div}D =0$, where $\operatorname{div}D
=\partial_{x}D(x)+\partial_{y}D(y) $. The image of a $K$-derivation $D$ with
$\operatorname{div}D=0$ of $K[x,y]$ does not need to be a Mathieu subspace (if $1 \notin \IM D$)~\cite{sun2017}. The Jacobian conjecture for $K[x,y]$ can also be reformulated as follows: a $K$-derivation $D$ of $K[x,y]$ which satisfies $1\in\IM D$ and
$\operatorname{div}D=0$ is locally finite~\cite[Conjecture 9.8.2]{now1994}. For locally finite derivations, van den Essen, Wright and Zhao proved the following result.

\begin{theorem}
\,\cite[Theorem 3.1]{vanden2010}\label{thewz} Let $K$ be a field of
characteristic $0$ and let $D$ be any locally finite $K$-derivation of
$K[x,y]$. Then $\IM D$ is a Mathieu subspace of $K[x,y]$.
\end{theorem}

The notion of Mathieu subspaces was introduced by Zhao in~\cite{zhao2010},
based on study on the Jacobian conjecture
 and inspired
by the Mathieu conjecture that implies the Jacobian conjecture
~\cite{mathieu1997}. The concept was subsequently extended to noncommutative algebras in
~\cite{zhao2012}. Mathieu subspaces are also called Mathieu-Zhao subspaces
as suggested by van den Essen~\cite{essen2014}. The
general references for Mathieu subspaces are~\cite{db,essenbk21,zhao2012}.

In this paper we will consider only the commutative case.

\begin{definition} \label{mz}
Let $A$ be a commutative $K$-algebra. A subspace $M$ of $A$ is called a
Mathieu subspace (or Mathieu-Zhao subspace) of $A$, if for all $f\in M$ such that $f^{m}\in M$ for all $m>0$
and for all $g\in A$, there exists an integer $n$ (depending on $f$ and $g$) such
that $f^{m}g\in M$ for all $m>n$.
\end{definition}

It is clear that ideals are Mathieu subspaces. Unlike ideals, Mathieu subspaces even in
 univariate polynomial algebras are not completely determined yet (see
~\cite{vanden2012,vanden2016}).

Various facts and problems in affine algebraic geometry are
related to Mathieu subspaces. A key issue is to prove that kernels and images
of some special linear maps such as derivations and more general differential
operators are Mathieu subspaces, though the verification of a
Mathieu subspace is generally difficult (see~\cite{essen2020} and~\cite[Ch.~5]{essenbk21},  and the references given there). Among others, $\mathcal{E}
$-derivations were also considered by Zhao~\cite{zhao2018}. A $K$-$\mathcal{E}
$-derivation of a $K$-algebra is a linear map $\delta$ such that $\id-\delta$
is an algebra homomorphism. Zhao formulated the following conjecture for
general associative algebras~\cite{zhao2018}, though we focus on the case of
polynomial algebras in this paper.

\begin{conjecture}
[The LFED Conjecture]Let $K$ be a field of characteristic zero. Images of
locally finite $K$-derivations and $K$-$\mathcal{E}$-derivations of
$K$-algebras are Mathieu subspaces.
\end{conjecture}

Recently, van den Essen and Zhao~\cite{vanden2018} found that both the case of  derivations and the case of $\mathcal{E}$-derivations of the LFED conjecture for Laurent polynomial algebras imply a remarkable theorem of Duistermaat and van der
Kallen~\cite{duikal}. This theorem states that the subspace consisting of Laurent polynomials
without constant terms is a Mathieu subspace of $\mathbb{C}[X,X^{-1}]$, where
$X$ denotes the $n$ variables $x_{1},x_2\dots,x_{n}$ and $X^{-1}$ denotes
$x_{1}^{-1},x_2^{-1},\dots,x_{n}^{-1}$.
The case of derivations of the LFED conjecture for Laurent polynomial algebras was proved earlier by Zhao~\cite{zhao20172}, using the theorem of Duistermaat and van der
Kallen.

The LFED conjecture has been proved for some special cases.
Zhao proved the LFED conjecture for finite dimensional algebras
~\cite{zhao20171} and for algebraic derivations and $\mathcal{E}$-derivations of
integral domains of characteristic zero~\cite{zhao2018}. The LFED conjecture
was also established for the Laurent polynomial algebra in one or two variables by
Zhao~\cite{zhao20172}, the field $K(X)$ of rational functions, the formal
power series algebra $K[[X]]$ and the Laurent formal power series algebra
$K[[X]][X^{-1}]$ by van den Essen and Zhao~\cite{vanden2018}.

The LFED conjecture for polynomial algebras is the most interesting case, but
few results are known. The LFED conjecture for the univariate polynomial algebra
$K[x]$ was proved by Zhao~\cite{zhao2017}. For $K[x,y]$, the case of
derivations was proved by van den Essen et al.\@ (Theorem~\ref{thewz}), but the case of $\mathcal{E}%
$-derivations remained unknown. For polynomial algebras in three variables, the
conjecture was verified for some locally nilpotent derivations
~\cite{liu2019,sun2021} and linear derivations and $\mathcal{E}$-derivations
~\cite{tian}. For general $K{[X]}$, the conjecture was proved for diagonal
derivations and $\mathcal{E}$-derivations, and for monomial-preserving derivations in~\cite{vanden2010,vanden2017}.

The aim of this paper is to present an analogue of Theorem~\ref{thewz} for
$\mathcal{E}$-derivations of $K[x,y]$ by proving the following result.

\begin{theorem}\label{thmain}Let $K$ be a field of characteristic $0$ and let $\delta$ be
any locally finite $K$-$\mathcal{E}$-derivation of $K[x,y]$. Then $\IM \delta$
is a Mathieu subspace of $K[x,y]$.
\end{theorem}

Theorem~\ref{thewz} and~\ref{thmain} together confirm the LFED conjecture for
polynomial algebras in two variables.

Throughout this paper $K$ denotes a field of characteristic $0$, and $X$
denotes the $n$ variables $x_{1},x_2,\dots,x_{n}$.

The rest of this paper is devoted to proving Theorem~\ref{thmain}. In Section 2, we
classify the locally finite endomorphisms of $\mathbb{C}[x,y]$ into seven
classes under conjugation.
The plan is to prove Theorem~\ref{thmain} by means of showing it
for the corresponding seven classes of $\mathcal{E}$-derivations. In Section 3, we
deal with the most complex case: the fourth class. In Section 4, we first reduce the LFED conjecture
of $K[X]$ to that of $\mathbb{C}[X]$. Then we finish the proof of Theorem~\ref{thmain} by examining the seven classes of $\mathcal{E}$-derivations individually.

\section{Classification of locally finite endomorphisms of $\mathbb{C}[x,y]$}

Let $F= (F_{1}, F_2,\dots, F_{n} )$ be a polynomial endomorphism of the affine
space $K^{n}$. Then there is a unique endomorphism $F^*$ of the polynomial
algebra $K[X]$ such that $F^*(x_{i})=F_{i}$ for $i=1,2,\dots,n$. Polynomial
endomorphisms of $K^{n}$ correspond one-to-one with endomorphisms of $K[X]$
under $F\mapsto F^*$, and ${(F\circ G)}^*=G^*\circ F^* $ for all
polynomial endomorphisms $F,G$ of $K^{n}$ (see~\cite{essenbk}).

A $K$-linear map $\phi$ of a $K$-vector space $V$ is called locally finite if
for each $v\in V$ the subspace generated by $\{\phi^{i}(v)\mid i\in \mathbb{N}\}$ is
finite dimensional~\cite[Definition 1.3.5(i)]{essenbk}.

According to~\cite[Definition 1.4 and Theorem 1.1]{fur2007} a polynomial
endomorphism $F$ of the affine space ${K}^{n}$ is locally finite if and only
if the endomorphism $F^*$ of $K[X]$ is locally finite.

By a result of Friedland and Milnor~\cite{Fri89}, Furter~\cite{fur1999} proved
that each locally finite polynomial automorphism of $\mathbb{C}^{2}$ is conjugate
to a triangular automorphism. Maubach~\cite[Lemma 2.16]{mau2015} classified,
up to conjugation by triangular automorphisms, the triangular automorphisms of
$K^{2}$ into two classes: affine and sequential.

Based on~\cite{fur1999,fur2007,mau2015}, we will classify the locally finite
endomorphisms of $\mathbb{C}[x,y]$ under conjugation into seven classes for
our purpose. We work with $\mathbb{C}$ in this section, thought some results
are valid for arbitrary fields of characteristic zero.

Denote by $\mathbb{N}$ the set of nonnegative integers and by $\mathbb{N}^{*}$
the set of positive integers. Let $K^*=K\setminus \{0\}$ for any field $K$.

\begin{theorem}\label{lem4forms} Let $\phi$ be a locally finite endomorphism of
$\mathbb{C}[x,y]$. Then up to conjugation $\phi$ satisfies one of the
following conditions:
\begin{enumerate}
\item $\phi(x)=bx$ and $\phi(y)=ay$, for $a,b\in\mathbb{C}^{*}$;\label{lem4forms1}

\item $\phi(x)=bx$ and $\phi(y)=y+1$, for $b\in\mathbb{C}^{*}$;\label{lem4forms2}

\item $\phi(x)=b^{s}x+ay^{s}$ and $\phi(y)=by$, where $s\in\mathbb{N}^{*}$,
$a\in\mathbb{C}$, $b\in\mathbb{C}^{*}$, and $b$ is not a root of unity;\label{lem4forms3}

\item $\phi(x)=b^{s}x+y^{s}p(y^{r})$ and $\phi(y)=by$, where $r\in
\mathbb{N}^{*}$, $s\in\mathbb{N}$, $b$ is a primitive $r$th root of unity, and
$p(y)\in\mathbb{C}[y]$ is monic;\label{lem4forms4}

\item $\phi^{2}=\phi^{3}$;\label{lem4forms5}

\item $\phi(x)=\lambda x+yg$ and $\phi(y)=0$, for $\lambda\in
\mathbb{C}^{*}$ and $g\in\mathbb{C}[x,y]$;\label{lem4forms6}

\item $\phi(x)=x+\lambda+yg$ and $\phi(y)=0$, for $\lambda\in
\mathbb{C}^{*}$ and $g\in\mathbb{C}[x,y]$.\label{lem4forms7}
\end{enumerate}
\end{theorem}

To prove the theorem, we need to generalize~\cite[Lemma 4.4]{fur2007} by removing
the assumption $F(0)=0$. Our proof follows that of~\cite[Lemma 4.4]{fur2007}.

\begin{lemma}\label{fenjie} Let $\phi$ be a locally finite endomorphism of $\mathbb{C}[x,y]$ that is not invertible. Then there exist homomorphisms $\mu:\mathbb{C}[z]\rightarrow\mathbb{C}[x,y]$ and $\nu:\mathbb{C}[x,y]\rightarrow\mathbb{C}[z]$ such that $\phi=\mu\nu$ and $\nu\mu(z)=az+b$ for some $a,b\in\mathbb{C} $.
\end{lemma}

\begin{proof}
By~\cite[Proposition 1.1]{fur2007} the Jacobian determinant $J(\phi
(x),\phi(y))=0$. By~\cite[Theorem 1.4]{now1988}, there exist $v_{1},v_{2}
\in\mathbb{C}[z]$ and $u\in\mathbb{C}[x,y]$ such that $\phi(x)=v_{1}(u)$ and
$\phi(y)=v_{2}(u)$. Let $\mu$ be the homomorphism from $\mathbb{C}[z]$ to
$\mathbb{C}[x,y] $ defined by $\mu(z)=u$, and let $\nu$ be the homomorphism
from $\mathbb{C}[x,y]$ to $\mathbb{C}[z] $ defined by $\nu(x)=v_{1}(z)$ and
$\nu(y)=v_{2}(z)$. Then $\mu\nu(x)=\mu(v_{1}(z))=v_{1}(\mu(z))=v_{1}
(u)=\phi(x)$. Similarly, $\mu\nu(y)= \phi(y)$. Hence $\phi=\mu\nu$.

Let $\nu\mu(z)=f(z)$.  Then $f(z)=u(v_1(z),v_2(z))$, and
\[
\phi^{n}(u)=\phi^{n}\mu(z)=\mu{(\nu\mu)}^{n}(z)=\mu(f^{n}(z))=f^{n}(u),
\]
where $f^n$ denotes the polynomial composition of $f$ with itself $n$ times. Suppose that $\deg f(z)>1$.
Then $\deg u(x,y)\geq1$ and $\deg\phi^{n}(u)={(\deg f(z))}^{n}\deg u $ for all
$n\in{\mathbb{N}}^*$, which implies that ${\{\deg\phi^{n} (u) \}}_{n\geq1}$ is unbounded. But that is not possible: since $\phi$ is locally finite, ${\{\deg\phi^{n} (u) \}}_{n\geq1}$ must be bounded. Therefore $\deg f(z)\leq1$, and so $\nu\mu(z)=f(z) =az+b$ for some $a,b\in\mathbb{C}$.
\end{proof}

We conclude this section with the proof of Theorem~\ref{lem4forms}.

\begin{proof}
[Proof of Theorem~\ref{lem4forms}]If $\phi$ is an automorphism,   $\phi$ is
conjugate to one of (1)--(4) by~\cite[Lemma 2.16]{mau2015}.

If $\phi$ is not invertible, then we have $\phi=\mu\nu$ and $\nu\mu(z)=az+b$
for some $a,b\in\mathbb{C}$, as in Lemma~\ref{fenjie}.

If $a=0$, then ${(\nu\mu)}^{2}=\nu\mu$. Thus $\phi^{3}=\mu{(\nu\mu)}^{2}\nu
=\mu(\nu\mu)\nu=\phi^{2}$. This is case (5).

If $a\neq0$, then $\nu\mu$ is an automorphism of $\mathbb{C}[z]$, which
implies $\nu$ is an epimorphism. Let $\pi:\mathbb{C}[x,y]\rightarrow
\mathbb{C}[z]$ be the epimorphism defined by $\pi(x)=z,~\pi(y)=0$. By
~\cite[Epimorphism theorem]{ab1975}, there exists an automorphism
$\delta:\mathbb{C}[x,y]\rightarrow\mathbb{C}[x,y]$ such that $\pi=\nu\delta$.
Let $\psi=\delta^{-1}\phi\delta$. Then
\[
\psi(y)=\delta^{-1}\phi\delta(y)=\delta^{-1}\mu\nu\delta(y)=\delta^{-1}\mu
\pi(y)=0.
\]
Write $\psi(x)=f+yg$ for some $f\in\mathbb{C}[x]$ and $g\in\mathbb{C}[x,y]$.
Then $\pi\psi^{n}(x)=f^{n}(z)$ for all $n\in\mathbb{N}$. Since $\psi$ is
locally finite, ${\{\deg\pi\psi^{n}(x)\}}_{n\geq 1}$ is bounded,
which implies $\deg f\leq1$. Thus $f=\lambda_{1}x+\lambda_{2}$ for some
$\lambda_{1} , \lambda_{2}\in\mathbb{C}$. Therefore,
\[
\psi(x)=\lambda_{1}x+\lambda_{2}+yg\ \text{and}\ \psi(y)=0,
\]
for $\lambda_{1},\lambda_{2}\in\mathbb{C}$ and $g\in\mathbb{C}[x,y]$.

If $\lambda_{1}=0$, then $\psi^{2}(x)=\lambda_2=\psi^{3}(x)$ and so $\phi^{2}=\phi^{3}$.
This is case (5).

Suppose that $\lambda_{1}\neq0$. We will distinguish several cases.

If $\lambda_{2}=0 $, then case (6) applies.

If $\lambda_{2}\ne0 $ and $\lambda_{1}=1 $, then case (7) applies.

To complete the proof, it only remains to consider the case $\lambda_{2}\ne0 $
and $\lambda_{1}\neq1$. Define the automorphism $\eta$ of $\mathbb{C}[x,y]$ by
$\eta(x)=x+{(1-\lambda_{1})}^{-1}\lambda_{2}$ and $\eta(y)=y. $ Then $\eta
\psi\eta^{-1}(x) =\lambda_{1}x+yg(x+{(1-\lambda_{1})}^{-1}\lambda_{2},y) $ and
$\eta\psi\eta^{-1}(y) =0 $. So case (6) applies.
\end{proof}

\section{Theorem \ref{thmain} for $\delta = \id - \phi$ where $\phi$ satisfies \ref{lem4forms}(\ref{lem4forms4})}

In this section, we first determine the image of $\delta = \id - \phi$ where $\phi$ satisfies \ref{lem4forms}(\ref{lem4forms4}). After that, we prove that this image is a Mathieu subspace of $\mathbb{C}[x,y]$. Furthermore, we generalize the second result to arbitary fields $K$ of characteristic $0$.

For $\beta =(\beta _{1},\beta _{2},\dots ,\beta _{n})\in
\mathbb{N}^{n}$, write $X^{\beta }$
for $x_{1}^{\beta _{1}}x_{2}^{\beta _{2}}\cdots x_{n}^{\beta _{n}}$.
Denote by $\LT(f)$ the leading term of   $f\in K[X]\setminus \{0\}$ with respect to a fixed  monomial well-ordering on $K[X]$ (see~\cite{Cox}).

\begin{lemma}\label{lemmaeta} Let $S$ be a subspace of $K[X]$ spanned by monomials in
$K[X]$ and $\eta:S\rightarrow K[X]$ be a $K$-linear map such that
$\eta(S)\subseteq S$. If
\begin{equation}\label{eqLT}
   \text{for all}~X^{\beta}\in S, \LT(\eta(X^{\beta}))=c_{\beta}X^{\beta}~\text{for some}~c_{\beta}\in K^{\ast}
\end{equation}
with respect to a fixed monomial well-ordering in $K[X]$, then $S=$ $\eta(S)$.
\end{lemma}

\begin{proof}
Suppose that, on the contrary, $S\setminus\eta(S)$ is not empty.
Then we  choose $X^{\alpha_{0}}$ being the least element in $S\setminus \eta(S)$  with respect
to a monomial well-ordering in $K[X]$. Since $\eta(S)\subseteq S$,  we can write
$\eta(X^{\alpha_{0}})=a X^{\alpha_{0}}+\sum_{i=1}^{m}a_{i}X^{\alpha_{i}}$ with $a \in
K$,  $a_{i}\in K^{\ast}$, and distinct $X^{\alpha_0},X^{\alpha_1},\dots,X^{\alpha_m}\in S$.  By (\ref{eqLT}),  we have $a\ne 0$ and $X^{\alpha_{0}}>X^{\alpha_{i}}$ for all
$1\leq i\leq m$. Then $X^{\alpha_{i}}\in\eta(S)$ for all $1\leq i\leq m$ by
the minimality of $X^{\alpha_{0}}$. Therefore, $X^{\alpha_{0}}=a^{-1}%
(\eta(X^{\alpha_0})-\sum_{i=1}^{m}a_{i}X^{\alpha_{i}})\in\eta(S)$, a contradiction.
\end{proof}

For a subset $S$ of a $K$-algebra $A$, denote by $\spn_{K} S $ and $\langle S\rangle_A$ the subspace and the ideal generated by $S$, respectively. We write $\langle S\rangle$ instead of $\langle S\rangle_{\mathbb{C}[x,y]}$.

\begin{lemma} \label{lastform1}
Let $\phi$ be the endomorphism in Theorem~\ref{lem4forms}(\ref{lem4forms4})
and $\delta=\id-\phi$.
Then
\begin{equation}\label{eqidinim}
\langle y^{s}p(y^{r}) \rangle \subseteq \IM\delta.
\end{equation}
\end{lemma}

\begin{proof}
By induction on $i$, one can show that $\phi^i (x) = b^{si}x + i b^{s(i-1)} y^s p(y^r)$ and $\phi^i (y) = b^i y$. So if we define $q = r b^{-s} y^s p(y^r)$, then
$$
\phi^r(x) = x + q~\text{and}~\phi^r(y) = y.
$$
Let $\delta' = \id - \phi^r$. Then $\delta' = \delta(\id+\phi+\phi^2+\cdots+\phi^{r-1})$, so $\IM \delta' \subseteq \IM\delta$. Therefore, it suffices to show that $\langle y^{s}p(y^{r})\rangle \subseteq \IM\delta'$.

Define a linear map $\eta:\langle x \rangle\rightarrow\mathbb{C}[x,y]$ by $\eta(x^{m}
y^{n})=x{q}^{-1}\delta'(x^{m}y^{n})$ for all $x^{m}y^{n}\in \langle x \rangle
$. Then
\[
\eta(x^{m}y^{n})=\frac{x}{q}\Big(x^my^n-\big(x+q\big)^my^n\Big)
=-\sum_{i=1}^{m}\binom{m}{i}x^{m-i+1}q^{i-1}y^n,
\]
where the sum is zero whenever its lower limit is bigger than its upper limit. This implies that $\eta(\langle x \rangle) \subseteq \langle x \rangle$, and that $\LT\big(\eta(x^{m}y^{n})\big) = -mx^m y^n$ with respect to the lex order in $\mathbb{C}[x,y]$. Thus by Lemma~\ref{lemmaeta},
$\eta(\langle x \rangle) = \langle x \rangle$. Hence
$$
\delta' (\langle x \rangle) = q x^{-1} \eta(\langle x \rangle) = q x^{-1} \langle x \rangle = \langle y^{s}p(y^{r}) \rangle,
$$
which gives the desired result.
\end{proof}

\begin{corollary} \label{lastform}
Let $\phi$ be the endomorphism in Theorem~\ref{lem4forms}(\ref{lem4forms4})
and $\delta=\id-\phi$.
Then
\begin{equation}\label{eqim=c+}
   \IM\delta=C+\langle y^{s}p(y^{r})\rangle,
\end{equation}
where $C=\spn_{\mathbb{C}}\{x^{m}y^{n}\mid m,n\in\mathbb{N}\text{ and }r\nmid ms+n\}$.
\end{corollary}

\begin{proof}
For any $m,n\in\mathbb{N}$, a direct computation shows that
$$
\delta(x^{m}y^{n}) \equiv 
(1-b^{ms+n})x^{m}y^{n} \pmod{\langle y^s p(y^r) \rangle}.
$$
Since $1 - b^{ms+n} = 0$, if and only if $r \mid ms + n$, we infer from lemma \ref{lastform1} that $\IM \delta$ is as given.
\end{proof}

So we have determined the image of $\delta = \id - \phi$ where $\phi$ satisfies \ref{lem4forms}(\ref{lem4forms4}). We advance with proving that this image is a Mathieu subspace of $\mathbb{C}[x,y]$. For that purpose, we first formulate and prove two lemmas.

Write $\cfx{f}{i}$ for the coefficient of $x^i$ of an element of $f \in K[x,y]$, viewed as polynomial over $K[y]$. Let $\bar{K}$ be an algebraic closure of $K$.

\begin{lemma} \label{lastform2}
Let $r$ be a positive integer, $s$ a nonnegative integer, and $p(y)\in K[y]$.
Suppose that $f \in K[x,y]$, such that $f^n \notin \langle y^s p(y^r) \rangle_{K[x,y]}$ for all $n \in \mathbb{N}$. Then there exist a factor $q$ of $y^sp(y^r)$, and an $i \in \mathbb{N}$, such that the following holds. Either $q = y$ or $q = y^r - \lambda$ for some $\lambda \in \bar{K}^*$, and
$$
q \nmid \cfx{f}{i} ~\text{and}~ q \mid \cfx{f^j}{ij} - \big({\cfx{f}{i}}\big)^j ~\text{for all $j > 0$}.
$$
\end{lemma}

\begin{proof}
Suppose first that $f(x,\alpha) = 0$ for all roots $\alpha \in \bar{K}$ of $y^sp(y^r)$. Take $j$ such that the multiplicity of every root of $y^sp(y^r)$ is bounded from above by $j$. Then $f^j \in \langle y^sp(y^r) \rangle$. Contradiction.

Suppose next that $f(x,\alpha) \ne 0$ for some root $\alpha \in \bar{K}$ of $y^sp(y^r)$. Then we can choose a factor $q$ of $y^sp(y^r)$ with this root $\alpha$, such that either $q = y$ or $q = y^r - \lambda$ for some $\lambda \in \bar{K}^*$. As $f(x,\alpha) \ne 0 = q(\alpha)$, there exists $i$ such that $q \nmid \cfx{f}{i}$. Take $i$ maximal (or minimal) as such. If we compute $\cfx{f^j}{k}$ modulo $q$ for all $k \ge ij$ ($k \le ij$) by induction on $j$ by way of discrete convolution, then we obtain that $\cfx{f^j}{k} \equiv 0 \pmod{q}$ for all $k > ij$ ($k < ij$) and $\cfx{f^j}{ij} \equiv \big(\cfx{f}{i}\big)^j \pmod{q}$, for all $j > 0$.
\end{proof}

\begin{lemma} \label{lastform34}
Let $r$ be a positive integer. Take $S = \{y,y^2,\ldots,y^{r-1}\}$, and suppose that $\lambda \in K$ and $h \in K[y]$, such that $h^j \in \spn_{K} S + \langle y^r - \lambda \rangle_{K[y]}$ for all $j > 0$. Then $h^r \in \langle y^r - \lambda \rangle_{K[y]}$.
\end{lemma}

\begin{proof}
An arbitrary element of $S$ is of the form $y^i$, where $0 < i < r$. Furthermore, the matrix of the multiplication by $y^i$ modulo $y^r - \lambda$ as a $K$-linear map, with respect to the basis $\mathcal{B}$ consisting of the residue classes of $1,y,y^2,\ldots,y^{r-1}$, is
$$
\left(\begin{array}{cc}
0_{i,r-i} & \lambda I_{i}  \\
I_{r-i} & 0_{r-i,i}
\end{array} \right),
$$
i.e.\@ a matrix with trace $0$.

Let $M$ be the matrix of the multiplication by $h$ modulo $y^r - \lambda$, with respect to basis $\mathcal{B}$. Then for all $j > 0$, the matrix $M^j$ of the multiplication by $h^j$ modulo $y^r - \lambda$ (with respect to basis $\mathcal{B}$) has trace $0$.  But that means that $M^r = 0$. So $h^r \cdot 1 \in \langle y^r - \lambda \rangle_{K[y]}$.
\end{proof}

We finally prove that the image $C+\langle y^{s}p(y^{r})\rangle$ in Corollary \ref{lastform} is a Mathieu subspace of $\mathbb{C}[x,y]$.

\begin{theorem}\label{C+p} Let $r$ be a positive integer, $s$ a nonnegative integer, and
$p(y)\in K[y]$. Take
\[
C = \spn_{K}\{x^{m}y^{n}\mid m,n\in\mathbb{N}\text{ and }r \nmid ms + n\}.
\]
Then $C+\langle y^{s}p(y^{r})\rangle_{K[x,y]}$ is a Mathieu subspace of $K[x,y]$.
\end{theorem}

\begin{proof}
Suppose that $f \in K[x,y]$ such that $f^n \in C+\langle y^{s}p(y^{r})\rangle_{K[x,y]}$ for all $n > 0$. Suppose first that $f^N \in \langle y^s p(y^r) \rangle_{K[x,y]}$ for some $N \in \mathbb{N}$. Then
$$
g f^m \in \langle y^s p(y^r) \rangle_{K[x,y]} \subseteq C+\langle y^{s}p(y^{r})\rangle_{K[x,y]} ~\text{for all}~ g \in K[x,y] ~\text{and all}~ m \ge N,
$$
in agreement with Definition \ref{mz}.

Suppose next that $f^n \notin \langle y^s p(y^r) \rangle_{K[x,y]}$ for all $n \in \mathbb{N}$. We will derive a contradiction. Take $q \mid y^s p(y^r)$ and $i \in \mathbb{N}$ as in Lemma \ref{lastform2}.

Suppose first that $q = y$. Then $f^r \in C + \langle y \rangle_{K[x,y]}$, so
$$
\cfx{f^r}{ir} \in \spn_{K}\{y^n \mid n \in \mathbb{N}\text{ and }r \nmid irs + n\} + \langle y \rangle_{K[y]} = \langle y \rangle_{K[y]}.
$$
From Lemma \ref{lastform2}, it follows that $y \mid \big(\cfx{f}{i}\big)^r$ and $y \nmid \cfx{f}{i}$, which is impossible.

Suppose next that $q = y^r - \lambda$ for some $\lambda \in \bar{K}^{*}$. Let $S = \{y,y^2,\ldots,y^{r-1}\}$. Then $f^{jr} \in C + \langle  y^r - \lambda \rangle_{\bar{K}[x,y]}$, so
\begin{align*}
\cfx{f^{jr}}{ijr} &\in \spn_{\bar{K}}\{y^n \mid n \in \mathbb{N}\text{ and }r \nmid ijrs + n\} + \langle y^r - \lambda \rangle_{\bar{K}[y]} \\
&= \spn_{\bar{K}} S + \langle y^r - \lambda \rangle_{\bar{K}[y]}, ~\text{for all $j > 0$}.
\end{align*}
Let $h = \big(\cfx{f}{i}\big)^r$. From Lemma \ref{lastform2}, we deduce that $y^r - \lambda \mid \cfx{f^{jr}}{ijr} - h^j$ for all $j > 0$, so $h^j \in \spn_{\bar{K}} S + \langle y^r - \lambda \rangle_{\bar{K}[y]}$ for all $j > 0$. On account of Lemma \ref{lastform34}, $h^r \in \langle y^r - \lambda \rangle_{\bar{K}[y]}$. So
$$
y^r - \lambda \mid h^r = \big(\cfx{f}{i}\big)^{r^2}.
$$
As $\gcd\{y^r - \lambda, \allowbreak ry^{r-1}\} = 1$, it follows that $y^r - \lambda$ is square-free. So $y^r - \lambda \mid \cfx{f}{i}$. This contradicts Lemma \ref{lastform2}.
\end{proof}

The proof of Theorem \ref{C+p} works for characteristic $> r$ as well. Theorem \ref{C+p} holds in addition if $r$ is a power of the characteristic of $K$, because $f^r \bmod y^sp(y^r) \in K[x^r,y^r]$ in that case. But Theorem \ref{C+p} does not hold if $r$ has a factor $u$ such that $r/u-1$ is a multiple of the characteristic of $K$. A counterexample can be constructed by taking $f = -(y^u+y^{2u}+\cdots+y^{r-u})$, using the fact that $y^r - 1 \mid f^2 - f$ under the given conditions.

\section{Proof of Theorem~\ref{thmain}}

Let $A$ be a $K$-algebra. By a $K$-derivation of $A$ we mean a $K$-linear map
$D:A\rightarrow A$ satisfying
\begin{equation}
D(ab)=aD(b)+D(a)b ~\text{for all $a,b\in A$.}
\label{Leibniz}
\end{equation}
By a $K$-$\mathcal{E}$-derivation of $A$ we mean a $K$-linear map
$\delta:A\rightarrow A$ such that $\phi = \id-\delta$ is an algebra endomorphism of
$A$ (see~\cite{zhao2018}). Notice that for such $\delta$,
\begin{equation}
\delta(ab) = \delta(a)b+ a\delta(b) - \delta(a)\delta(b) ~\text{for all $a,b \in A$.} \label{ELeibniz}
\end{equation}
In literature,  $\mathcal{E}$-derivations are also called   skew
derivations or   $\phi$-derivations (with $\phi=\id-\delta$) (see~\cite{Bresar2002,Kharchenko1992}).

Without causing misunderstanding, we will write $\mathcal{E}$-derivations
instead of $K$-$\mathcal{E}$-derivations.

An $\mathcal{E}$-derivation $\delta$ of $A$ is locally finite if and only if the
associated endomorphism $\phi=\id-\delta$ is locally finite, because $\spn_K\{f,\delta(f),\delta^2(f),\ldots,\delta^j(f)\} =  \spn_K\{f,\phi(f),\phi^2(f),\ldots,\phi^j(f)\}$ for all $f \in A$ and all $j \in \mathbb{N}$.

Let $D$ be a derivation or $\mathcal{E}$-derivation of $K[X]$. If we combine the $K$-linearity of $D$ with \eqref{Leibniz} or \eqref{ELeibniz} respectively, then we can infer that $D$ is
uniquely determined by $D(x_1),D(x_2),\ldots,D(x_n)$, and that $D$ is locally finite if and only if
$\spn_{K}\{x_{i},D(x_{i}),D^{2}(x_{i}),\ldots\}$ is finite dimensional for all
$i=1,2,\ldots,n$.

%

\begin{lemma}\label{lemext} Let $L\subseteq K$ be a field extension, $D$ a $K$-$\mathcal{E}%
$-derivation (resp.\@ $K$-derivation) of $K[X]$ and $D_{L}$ an $L$-$\mathcal{E}%
$-derivation (resp.\@ $L$-derivation) of $L[X]$ such that $D(x_{i})=D_{L}(x_{i})
$ for $i=1,2,\ldots,n$.
\begin{enumerate}
\item $D$ is locally finite if and only if $D_{L}$ is locally finite;
\item If $\IM D$ is a Mathieu subspace of $K[X]$ then $\IM D_{L}$ is a
Mathieu subspace of $L[X]$.
\end{enumerate}
\end{lemma}

\begin{proof}
We may assume that $K[X]=K\otimes_{L}L[X]$. Then $D=\id_{K}\otimes_{L} D_{L}$ and
$\IM D=K\otimes_{L}\IM D_{L}$. It follows that
\[
\spn_{K}\{x_{i},D(x_{i}),D^{2}(x_{i}),\dots\}=K\otimes_{L}\spn_{L}%
\{x_{i},D_{L}(x_{i}),D_{L}^{2}(x_{i}),\dots\}
\]
for $i=1,2,\ldots,n$. Thus $D$ is locally finite if and only if $D_{L}$ is
locally finite, and so (1) follows. (2) follows from Lemma~\ref{lemtensor} below.
\end{proof}

\begin{lemma}
\,\cite[Lemma 2.5]{vanden2010}\label{lemtensor} Let $L\subseteq K$ be a field
extension, $A$ an algebra over $L$, and $M$ an $L$-subspace of $A$. Assume
that $K\otimes_{L}M$ is a Mathieu subspace of the $K$-algebra $K\otimes_{L}A$.
Then $M$ is a Mathieu subspace of the $L$-algebra $A$.
\end{lemma}

In~\cite{vanden2018} the LFED conjecture of a $K$-algebra is reduced to the
case of a $\bar K$-algebra, where $\bar{K}$ is an algebraic closure of $K$. We
now reduce the LFED conjecture of $K[X]$ to the case of $\mathbb{C}[X]$.

\begin{lemma}\label{LFEDC} If the LFED conjecture holds for $\mathbb{C}[X]$, then it holds for $K[X]$ over any field $K$ of characteristic $0$.
\end{lemma}

\begin{proof}
We prove this lemma by contraposition. Assume that $D$ is a counterexample to the LFED conjecture for $K[X]$. Then $D$ is either a locally finite $K$-derivation or a locally finite $K$-$\mathcal{E}$-derivation of $K[X]$, and $\IM D$ is not a Mathieu subspace of $K[X]$. So there exist $f,g\in K[X]$ and positive integers $m_{1}<m_{2}<\cdots$ such that $f^{m}\in\IM D$ for all $m>0$ and $f^{m_{i}}g\notin\IM D$ for all $i=1,2,\ldots$.

Let $L$ be the subfield of $K$ generated by the coefficients of $f,g$ and
$D(x_{i})$ for $i=1,2,\ldots,n$. Then the restriction of $D$ to $L[X]$,
denoted by $D_{L}$, is an $L$-($\mathcal{E}$\discretionary{-)}{}{-)}derivation of $L[X]$ since $D(x_{i})\in L[X]$. By the supposition, $\IM D_{L}$ is not a Mathieu subspace of
$L[X]$.

Since $L$ is a finitely generated extension of $\mathbb{Q}$, there
exist a subfield $L^{\prime}$ of $\mathbb{C}$ and an isomorphism $\sigma:L\to
L^{\prime}$ by~\cite[Lemma 1.1.13]{essenbk}. The field isomorphism
$\sigma$ can be extended to a ring isomorphism from $L[X]$ to $L^{\prime}[X]$
in a natural way, still denoted by $\sigma$, which is a semi-linear mapping
relative to $\sigma$. It follows that $D_{L^{\prime}}:=\sigma D_{L}\sigma
^{-1}$ is an $L^{\prime}$-($\mathcal{E}$\discretionary{-)}{}{-)}derivation of $L^{\prime}[X]$ and $\IM D_{L^{\prime}}$ is not a Mathieu subspace of $L^{\prime}[X]$.

There exists a
unique $\mathbb{C}$-($\mathcal{E}$\discretionary{-)}{}{-)}derivation of $\mathbb{C}[X]$, denoted by $D_{\mathbb{C}}$, such that $D_{\mathbb{C}}(x_{i})=D_{L^{\prime}}(x_{i})$  for $i=1,2,\ldots,n$. By Lemma~\ref{lemext}, $D_{\mathbb{C}}$ is locally finite just like $D$, and $\IM D_{\mathbb{C}}$ is not a Mathieu subspace of $\mathbb{C}[X]$. So the LFED conjecture does not hold for $\mathbb{C}[X]$.
\end{proof}

We conclude with the proof of Theorem~\ref{thmain}.

\begin{proof}
[Proof of Theorem~\ref{thmain}]By Lemma~\ref{LFEDC}, we may assume $K=\mathbb{C}$.  Let $\delta=\id-\phi$, where $\phi$ is an
endomorphism of $\mathbb{C}[x,y]$. The proof splits into seven cases according to Theorem~\ref{lem4forms}, due to the fact that $\IM\delta$ is a Mathieu subspace if and only if so is the image of $\sigma^{-1}\delta\sigma=\id-\sigma^{-1}\phi\sigma$ for any automorphism $\sigma$ of $\mathbb{C}[x,y]$.

Case (\ref{lem4forms1}). This case follows from~\cite[Corollary 4.4]{tian}.

Case (\ref{lem4forms2}). Since $1=\delta(-y)\in\IM\delta$, we can see that $\IM\delta$ is a Mathieu subspace of $\mathbb{C}[x,y]$ by~\cite[Proposition 1.4]{zhao20181}.

Case (\ref{lem4forms3}). For all $m,n\in\mathbb{N}$, we have
\[
\delta(x^{m}y^{n})=(1-b^{ms+n})x^{m}y^{n}-\sum_{i=1}^{m}\binom{m}{i}%
a^{i}b^{(m-i)s+n}x^{m-i}y^{is+n}.
\]
As $\delta(1) = 0$ and $s > 0$, $\IM\delta \subseteq \langle x,y  \rangle$ follows.

Assume that $m \ne 0$ or $n \ne 0$. 
Then $b^{ms+n}\neq1$, since $ms + n > 0$ and $b$ is not a root of unity. Fix the lex order in $\mathbb{C}[x,y]$. Then
\[
\LT(\delta(x^{m}y^{n}))=(1-b^{ms+n})x^{m}y^{n}.
\]
Since $\delta(\langle x,y  \rangle)\subseteq\langle x,y  \rangle$, by Lemma~\ref{lemmaeta} $\IM\delta=\langle x,y  \rangle$. Thus $\IM\delta$ is a Mathieu subspace of $\mathbb{C}[x,y]$.

Case (\ref{lem4forms4}). This case follows from Corollary~\ref{lastform} and Theorem~\ref{C+p}.

Case (\ref{lem4forms5}). This case follows from~\cite[Proposition 6.8]{zhao2018}.

Case (\ref{lem4forms6}). For $m,n\in\mathbb{N}$, not both zero, we have
\begin{numcases}{\delta(x^{m}y^{n})=}
      x^{m}y^{n}, & if $n>0$,\label{eq1}\\
     (1-\lambda^{m}) x^{m}+yf_{m}, &  if $m>0$ and $n=0$,\label{eqnumcase2}
    \end{numcases}
for some $f_m\in \mathbb{C}[x,y]$. As $\delta(1) = 0$, $\IM\delta \subseteq \langle x,y  \rangle$ follows. By (\ref{eq1}), we have
\begin{equation}\label{eqyinim}
   \langle y \rangle \subseteq\IM\delta.
\end{equation}

Suppose first that $\lambda$ is not a root of unity. Then   (\ref{eqnumcase2}) and (\ref{eqyinim}) yield  $x^{m}\in\IM\delta$ for all $m>0$.
Hence we get $\IM\delta=\langle x,y  \rangle$. In particular, $\IM\delta$ is a Mathieu subspace of $\mathbb{C}[x,y]$.

Suppose next that $\lambda$ is  an $r$th root of unity for some $r\in\mathbb{N}^*$. By induction on $i$, it follows that $\phi^i(x) = \lambda^i x + \lambda^{i-1}yg$ and $\phi^i(y) = 0$ for all $i \ge 1$. Hence $\phi^{r+1}(x) = \lambda x + y g$, and therefore $\phi^{r+1} = \phi$. So we can apply~\cite[Proposition 6.8]{zhao2018} again.

Case (\ref{lem4forms7}). Since
\begin{align*}
\delta\big({-\lambda^{-1}(x+yg)}\big)
&= -\lambda^{-1}(x+yg) + \lambda^{-1}\big(\phi(x)+\phi(y)\phi(g)\big) \\
&= -\lambda^{-1}(x+yg) + \lambda^{-1}(x+\lambda+yg) = 1,
\end{align*}
$\IM\delta$ is a Mathieu subspace of $\mathbb{C}[x,y]$ by~\cite[Proposition 1.4]{zhao20181}.
\end{proof}

\section*{Acknowledgements}

The authors are very grateful to Michiel de Bondt for making many valuable suggestions.

This work is supported by NSF of China (No. 12171194).

\bibliographystyle{abbrvnat}
\bibliography{ref}

\end{document}